\documentclass[11pt]{article}

\usepackage{amsfonts}
\usepackage{amssymb}
\usepackage{amsmath}
\usepackage{fancyhdr}
\usepackage{eufrak}
\usepackage{latexsym}
\usepackage{amsbsy}
\usepackage[all]{xy}
\usepackage[latin1]{inputenc}
\usepackage[T1]{fontenc}        
\usepackage{amsmath}

\newenvironment{theorem}[1]
{\vskip 2mm\noindent \textbf{Theorem #1}  \it}{\vskip 2mm}

\newenvironment{corollary}[1]
{\vskip 2mm\noindent \textbf{Corollary #1}  \it}{\vskip 2mm}

\newtheorem{defn}{Definition}[section]
\newtheorem{thm}[defn]{Theorem}
\newtheorem{prop}[defn]{Proposition}
\newtheorem{lem}[defn]{Lemma}
\newtheorem{cor}[defn]{Corollary}
\newtheorem{rem}[defn]{Remark}
\newtheorem{note}[defn]{Constraints}

\newcommand{\proof}{\vskip 2mm \noindent {\textsc{Proof : }}\rm}

\newcommand{\fin}{\hfill{\Large$\Box$}\\}
\newcommand{\finsec}{\hfill{\Large$\Box$}}

\newcommand{\al}{\alpha}
\newcommand{\aire}{Area \,}

\newcommand{\si}{\sigma}

\newcommand{\om}{\omega}\newcommand{\epsi}{\epsilon}

\newcommand{\C}{\mathbb {C}}
\newcommand{\R}{\mathbb {R}}
\newcommand{\N}{\mathbb {N}}
\newcommand{\Z}{\mathbb {Z}}

\newcommand{\Pj}{\mathbb {P}}
\newcommand{\cdb}{\textsf {Card }}
\newcommand{\orb}{\textsf {O}}

\newcommand{\Id}{{\rm Id}}

\newcommand{\Lip}{{\rm Lip \, }}
\newcommand{\Jac}{{\rm Jac \,  }}

\newcommand{\supp}{{\rm supp \, }}
\newcommand{\D}{\mathbb {D}}

\def\abs#1{\vert #1\vert}

\def\RRR{{\mathfrak R}}

\def\EE{{\cal E}}

\def\AA{{\cal A}}

\def\CC{{\cal C}}
\def\PP{{\cal P}}

\def\KK{{\cal K}}
\def\FF{{\cal F}}

\def\RR{{\cal R}}

\def\LL{{\cal L}}

\def\MM{{\cal M}}

\def\WW{{\cal W}}

\def\com{\ar@{}[rd]|{\circlearrowleft}}

\title {A distortion theorem for the iterated inverse branches of a holomorphic endomorphism of $\Pj^k$}
\author{Fran\c cois Berteloot and Christophe Dupont}
\date{ \today }

\begin{document}

\maketitle
 
\begin{abstract} We linearize the inverse branches of the iterates of holomorphic endomorphisms of $\Pj^k$ and thus overcome the lack of Koebe distortion theorem
in this setting when $k\ge 2$. We review several applications of this result in holomorphic dynamics.\\


\end{abstract}

\section{Introduction}

Let $f : \Pj^k \to \Pj^k$ be a holomorphic endomorphism of algebraic degree $d \geq 2$. This is a ramified covering of $\Pj^k$ of  degree $d^k$.  The equilibrium measure $\mu$ of $f$  is a mixing $f$-invariant measure, see  \cite[Section 1.3]{DS}. The Lyapunov exponents of $(f,\mu)$ are  larger than or equal to $\log \sqrt d$, see \cite{BD} or \cite[Section 1.7]{DS}. Setting $J:=\textrm{Supp}\;\mu$, we  thus have a non-uniform hyperbolic dynamical system $(J,f,\mu)$. 

The aim of the present article is to provide a substitute to the Koebe distortion theorem, which is only valid for $k=1$, to control the geometry of the iterated inverse branches of $f$ on $\Pj^k$. Our proof is based on a normalization of $f$ along $\mu$-typical backward orbits. Although it should face a major difficulty  due to resonances on the Lyapunov spectrum, we use a simple trick to get rid of all possible resonances and stay in the setting of linear normalizations.

Our Distortion Theorem allows to skip the use of delicate results on the Lie group structure of resonant maps, see \cite{GK, JV,KS, BDM}. This is a typical feature of our approach. At the end of the article we review various applications which illustrate how it can be used. \\

Let us now introduce our framework. To deal with inverse branches, we introduce the natural extension of $f$. Let $$\orb := \{ \hat x = (x_n)_{n \in \Z} \, \colon \,  x_{n+1} = f (x_n) \} $$ be the set of orbits of $f$ and let $\tau : \orb \to \orb$ be the right shift sending 
  $(\cdots,x_{-1},x_0,x_1,\cdots)$ to
$(\cdots,x_{-2},x_{-1},x_0,\cdots)$. We say that a function $\phi_\epsi : \orb \to ]0,1]$ (resp. $\orb \to [1,+\infty[$) is
    \emph{$\epsi$-slow} (resp.  \emph{$\epsi$-fast}) if 
    $$\forall \hat x \in \orb \ , \  e^{- \epsilon} \phi_\epsi (\hat x) \leq \phi_\epsi(\tau(\hat x)) \leq e^{
      \epsilon} \phi_\epsi (\hat x) . $$
Similarly, a sequence $(\delta_n)_{n \in \Z}$ in $]0,1]$ is \emph{$\epsilon$-slow} if $e^{-\epsilon} \delta_n \leq \delta_{n+1} \leq e^{\epsilon} \delta_n$ for every $n \in \Z$. Let $\hat \mu$ be the unique $\tau$-invariant measure on $\orb$ satisfying $\pi_*\hat \mu = \mu$, where $\pi(\hat x) := x_0$,  see  \cite[Chapter 10]{CFS}. The measure $\hat \mu$ is mixing as $\mu$. Since the equilibrium measure $\mu$ is a Monge-Amp\`ere mass with bounded local potentials, it does not give mass to the critical set $\CC$ of $f$. In particular, the $\tau$-invariant subset 
$$X := \{ \hat x = (x_n)_{n
  \in \Z} \, \colon \,  x_n \notin \CC \ , \
\forall n \in \Z \}  $$ 
satisfies $\hat \mu(X)=1$. For every $\hat x \in X$, we denote $f^{-n}_{\hat x}$ the inverse branch of $f^n$ which sends $x_0$ to $x_{-n}$, it is defined in a neighbourhood of $x_0$. We denote $dist$ the distance on $\Pj^k$ induced by the Fubini-Study metric and $B_x(r)$ the ball centered at $x$ of radius $r$.

\begin{theorem}{A :}\label{thmA} 
Let $f$ be a holomorphic  endomorphism of $\Pj^k$ and let $\mu$ be its equilibrium measure. Let $\Lambda_l < \cdots < \Lambda_1$ be the distinct Lyapunov  exponents of $(f,\mu)$ and let  $ k_j \ge 1$ be their multiplicities.

\noindent Then for every  $\epsilon  < \gamma \ll \Lambda_l $ and for   $\hat \mu$-almost every $\hat x \in X$, there exist

- an integer $n_{\hat x} \geq 1$ and real numbers $h_{\hat x}\ge 1$ and  $0 < r_{\hat x} , \rho_{\hat x} \le 1$,

- a sequence $(\varphi_{\hat x,n})_{n \geq 0}$ of injective holomorphic maps 

$$ \varphi_{\hat x,n} : B_{x_{-n}} (r_{\hat x} e^{-n(\gamma + 2\epsilon)}) \to \D^k(\rho_{\hat x} e^{n \epsi}) $$ 

sending
$x_{-n}$ to $0$  and satisfying

$$  e^{n( \gamma - 2\epsi )} dist(u,v) \leq \vert  \varphi_{\hat x,n} (u) - \varphi_{\hat x,n} (v)   \vert  \leq  e^{n (\gamma + 3 \epsi)} \, h_{\hat x} \, dist(u,v) , $$
  
- a sequence of linear maps  $(D_{\hat x,n})_{n \geq 0}$ which stabilize each
$$L_j := \{ 0 \} \times \cdots \times \C^{k_j}  \times \cdots \times \{0 \} , $$ 

satisfying

  \[  \forall v \in L_j  \ \ ,Ê\ \ e^{- n\Lambda_j    + n (\gamma - \epsilon) } \abs{v}
\leq \abs{D_{\hat x,n}(v)}  \leq e^{- n \Lambda_j + n (\gamma +  \epsilon) }  \abs{v} , \]

 for which the following diagram commutes for every $n \geq n_{\hat x}$:

\[  \xymatrix{
  B_{x_0}(r_{\hat x}) \ar[rrr]^{ f^{-n}_{\hat x}  } \ar[d]_{\varphi_{\hat x , 0}}& & &
      B_{x_{-n}}(r_{\hat x} e^{-n(\gamma + 2 \epsilon)}) \ar[d]^{\varphi_{\hat x , n}}     \\
      \D^k(\rho_{\hat x})  \ar[rrr]^{ D_{\hat x,n} }  & & & \D^k (\rho_{\hat x} e^{n\epsi}) . }   \]
Moreover, the fonctions  $\hat x \mapsto h^{-1}_{\hat x}, r_{\hat x}, \rho_{\hat x}$ are measurable and $\epsilon$-slow on $X$.\\

\end{theorem}

The following corollary shows how Theorem A can be used to estimate the convexity defect of the inverse branches
$f_{\hat x}^{-n}\left(B_{x_0}(r_{\hat x})\right)$. A special case of such a property was recently put forward in \cite{BB}. 

\begin{corollary}{(Convexity defect)}\label{Cor} 
Let $f$ be a holomorphic  endomorphism of $\Pj^k$ and let $\mu$ be its equilibrium measure. We keep the notations of Theorem A. Let $r'_{\hat x}:= r_{\hat x} / h_{\hat x}$ and, for every $0<t\le 1$, let 
$$E_{\hat x}^{-n}(t):=f_{\hat x}^{-n}\left(B_{x_0}(t r'_{\hat x})\right) \subset {\widetilde E}_{\hat x}^{-n}(t):=f_{\hat x}^{-n}\left(B_{x_0}(t r_{\hat x})\right) .$$
 Then for every pair of points
$p,q\in E_{\hat x}^{-n}(t)$ there exists a smooth path  connecting $p$ and $q$ in ${\widetilde E}_{\hat x}^{-n}(t)$
and whose length is smaller than $e^{5n\epsilon} h_{\hat x} d(p,q)$.
\end{corollary}

Let us finally mention that Theorem A  and its Corollary remain valid for every invariant ergodic measure $\nu$ with positive Lyapunov exponents $\Lambda_l' < \cdots < \Lambda_1'$ and, in particular,  when $\nu$ has metric entropy  $h(\nu) > \log d^{k-1}$, see \cite{dT,D2}. 

\section{Proof of Theorem A}

In this Section, we explain how to deduce Theorem A  from a linearization statement for chains of holomorphic contractions (Theorem B in Subsection \ref{contrholo}).

\subsection {The bundle map over $X$ generated by the inverse branches of $f$} \label{affchart}

We recall in Section \ref{app} basic definitions and facts concerning bundle maps. Let $(\psi_x)_{x \in \Pj^k}$ be a collection of  affine charts $\psi_x :\C^k \to \Pj^k$   with uniform bounded distortion and satisfying $\psi_x(0)=x$ (for instance fix an affine chart and rotate it thanks to unitary automorphisms of $\Pj^k$). To simplify the exposition, we shall ignore the distortions induced by these charts. For every $\hat x \in X$, let  $E_{\hat x}  := \{  \hat x \} \times \C^k$ and let $\psi_{\hat x} : E_{\hat x} \to \Pj^k$ defined by 
$$\psi_{\hat x} (\hat x, v) = \psi_{x_0} (v) .$$ 
Let $\hat f := \tau^{-1}$ be the left shift  
  $(\cdots,x_{-1},x_0,x_1,\cdots) \mapsto (\cdots,x_{0},x_{1},x_2,\cdots)$ on $X$, and let 
\[ F_{\hat
 x} := \psi^{-1}_{{\hat f(\hat x)}} \circ f \circ \psi_{{\hat x}} . \]
Since $f$ is continuous on $\Pj^k$ there exist constants $M_1 \leq M_0$ such that the bundle map 
\begin{displaymath} 
 \FF :
\begin{array}{rccc}
 &  E(M_1)          &  \longrightarrow    &         E(M_0)   \\
 &  (\hat x , v )   & \longmapsto &  \left(\hat f(\hat x), F_{\hat
 x}(v)  \right) 
\end{array}                                      
\end{displaymath}
is well defined. Note that $F_{\tau(\hat
 x)}$ is invertible near $0$ since $\hat x \in X$ implies $x_{-1} \notin \CC$.
Let  
$$ F^{-1}_{\hat x}:=  ( F_{\tau(\hat x)})^{-1} = \left( \psi^{-1}_{\hat x} \circ f \circ \psi_{\tau(\hat x)} \right)^{-1} . $$ 
The following lemma specifies the domains of definition of  these mappings.

\begin{lem}\label{abd}
For every $\epsilon >0$, there exists a $\epsilon$-slow function $\rho_\epsilon : X \to ]0,1]$ such that the bundle map
\begin{displaymath} 
 \FF^{-1} :
\begin{array}{rccc}
 &  E(\rho_\epsilon)          &  \longrightarrow    &         E(M_1)   \\
 &  (\hat x , v )   & \longmapsto &  \left( \tau(\hat x), F^{-1}_{\hat x}(v) \right)
\end{array}                                      
\end{displaymath}  
is well defined. 
\end{lem}

\proof Let $t(\hat x) := \vert  ( d_{x_{-1}} f  )^{-1} \vert^{-2}$. There
exists $c >0$ depending on the first and second
derivatives of $f$ on $\Pj^k$ such that  $F^{-1}_{\hat x}$ exists on
$E_{\hat x}(c \, t(\hat x))$ (see \cite[Lemma 2]{BD}). We let $\rho :=
\min \{ c\,  t, 1 \}$. Since $\log \rho$ is $\hat \mu$-integrable (see \cite[Theorem A.31]{DS}), there exists a $\epsi$-slow function $\rho_\epsi : X \to ]0,1]$ such that $\rho_\epsi \leq \rho$ (see \cite[Lemme 2.1]{BDM}). \finsec

\subsection {Oseledec-Pesin $\epsilon$-reduction}\label{OR}

We recall that $\Lambda_l < \cdots < \Lambda_1$ denote the distinct Lyapunov exponents of $(f,\mu)$ and that $k_j \geq 1$ denote their multiplicities. For every $j \in \{ 1,\cdots, l \}$ we set 
$$\hat L_j :=  \cup_{\hat x \in X}  \{ \hat x \}  \times \left[ \{ 0 \} \times \dots \times \C^{k_j}  \times \cdots \times \{0 \} \right]  =  \cup_{\hat x \in X}  \{ \hat x \}  \times   L_j  , $$ so that $\cup_{\hat x \in X}  \{ \hat x \} \times \C^k$ is equal to $\hat L_1 \oplus \cdots \oplus \hat L_l$. The Oseledec-Pesin's theorem may be stated as follows, see \cite[Theorem S.2.10, page 666]{KH}.

\begin{thm}\label{OPR}
Let $d_0 \FF^{-1}$ be the linear part of $\FF^{-1}$. For every $\epsilon > 0$ there exist an invertible linear bundle map $\CC_\epsi$ over $\Id_X$ and a $\epsi$-fast function $h_\epsi : X \to [1,+\infty[$ such that 
\begin{enumerate}
\item the linear bundle map $\AA  := \CC_\epsi \circ d_0 \FF^{-1} \circ \CC_\epsi^{-1}$ satisfies for every $1\leq j \leq l$ :
\[ \ \AA(\hat L_j) = \hat L_j \ \ \textrm{and} \ \ \forall (\hat x,v) \in \hat L_j \ , \   e^{- \Lambda_j - \epsilon} \abs{v}
\leq \abs{A_{\hat x}(v)}  \leq e^{-\Lambda_j + \epsilon}  \abs{v} . \]

\item $\forall \hat x \in X$, $\forall v \in \C^k$, $\abs{v} \leq
\abs {C_{\epsi, \hat x} (v)}  \leq  h_\epsi(\hat x) \abs{v}$.
\end{enumerate}
\end{thm}

We conjugate the bundle map $\FF^{-1}$ as follows:
$$ \WW := \CC_\epsi \circ \FF^{-1} \circ \CC_\epsi^{-1}  , $$ 
its expansion is of the form $\WW(\hat x , v) := (\tau(\hat x) , W_{\hat x}(v) )$, where $W_{\hat x} = C_{\epsi,\tau(\hat x)} \circ F^{-1}_{\hat x} \circ C_{\epsi, \hat x}^{-1}$.  Since $\FF^{-1} : E(\rho_\epsi) \to E(M_1)$ by Lemma \ref{abd}, Item 2 of Theorem \ref{OPR} implies that
$$\WW : E(\rho_\epsilon) \to E(M_1 h_\epsilon) . $$
By using $d_0 \WW = \AA$ and applying Lemma \ref{NN} on tame bundle maps with $\epsi' = \epsi$, we obtain a $\epsilon$-slow function, still denoted $\rho_\epsilon$, such that 
\begin{equation}\label{doublev}
\WW : E(\rho_\epsilon) \to E(\rho_\epsilon) \ \ \textrm{ and } \ \ \Lip \WW \leq e^{-\Lambda_l + 2\epsilon} . 
\end{equation}

\subsection {Resonances and constraints on $\gamma, \epsilon$}\label{RC}

As before, $\Lambda_l < \cdots < \Lambda_1$ denote the  distinct Lyapunov exponents of $(f,\mu)$ and $k_j \geq 1$ denote their multiplicities. One defines the  $k$-tuple
$$(\lambda_i)_{1\leq i \leq k} := (\Lambda_1, \cdots , \Lambda_1 , \cdots , \Lambda_l , \cdots , \Lambda_l)$$ 
by repeating $k_j$ times $\Lambda_j$.  For $j \in \{ 1, \cdots , l \}$, one defines the set of $j$-resonant indices by: 
\[ \RRR_j := \{ \al \in \N^k\;\colon\;
  \abs{\al} \geq 2 \  \textrm{ and }  \  \al_1\lambda_1 +\cdots +\al_k \lambda_k =
  \Lambda_j \} .  \] 
Since $\Lambda_l < \cdots < \Lambda_1$, one immediately sees that  one has  $2 \leq \abs{\al} \leq [\Lambda_j / \Lambda_l]  \leq [\Lambda_1 / \Lambda_l]$ and $\al_1 = \cdots =  \al_{k_1 + \cdots + k_j} = 0$ for every $\al \in \RRR_j$, where $[\cdot]$ stands for the entire part. \\

For each $\gamma > 0$ one defines a shifted Lyapunov spectrum by setting: $$ \Lambda_j^\gamma := \Lambda_j - \gamma . $$
Likewise,  one defines the $k$-tuple $(\lambda^\gamma_i)_{1\leq i \leq k}$ by repeating $k_j$ times the $\Lambda^\gamma_j$:
$$(\lambda^\gamma_i)_{1\leq i \leq k} = (\Lambda^\gamma_1, \cdots , \Lambda^\gamma_1 , \cdots , \Lambda^\gamma_l , \cdots , \Lambda^\gamma_l)$$
and one denotes
\[ \RRR^\gamma_j := \{ \al \in \N^k\;\colon\;
  \abs{\al} \geq 2 \  \textrm{ and}  \  \al_1\lambda^\gamma_1 +\cdots +\al_k \lambda^\gamma_k =
  \Lambda^\gamma_j \} .  \]

We now  fix  a constant $0 < a < \ln 4$  such that 
\begin{equation}\label {rol} 
 \al_1\lambda_1 +\cdots +\al_k \lambda_k  - \Lambda_j  \notin [-a,a] 
\end{equation}
for every $j \in \{ 1 , \cdots , l \}$ and for every $\al \in \N^k\setminus \RRR_j$  satisfying $2 \leq \abs{\al} \leq [2 \Lambda_1 / \Lambda_l]$.\\

One easily checks that  $\cup_{j=1}^l \RRR^\gamma_j =\emptyset$ if $\gamma$ is small enough, hence there is no resonance relation for the shifted Lyapunov spectrum.
More precisely, setting 
$$ b:= {1 \over 2} \min \{ \gamma , a  \} ,$$
one has $\al_1 \lambda^\gamma_1 +\cdots +\al_k \lambda^\gamma_k - \Lambda^\gamma_j  \notin [-b , b]$  for  $j \in \{ 1 , \cdots , l \}$ and   $2 \leq \abs{\al} \leq [\Lambda^\gamma_1 / \Lambda^\gamma_l]$ (see Lemma \ref{marge} in the Appendix). We actually impose the following precise Constraints on $\gamma$ and $\epsilon$, they will play an important role in our future estimates.
We stress that these constraints only depend  on the Lyapunov exponents $(\Lambda_j)_j$. 

\begin{note} \label{nota} We shall assume that $\gamma,\epsilon$ satisfy the following properties.

\begin{itemize} 
\item[1.]  The number  $\gamma >0$ is fixed and sufficently small so that:

$\gamma < \Lambda_l / 2 \ \ , \ \ \gamma ( [ \Lambda^\gamma_1 / \Lambda^\gamma_l ] -1)< a/2 \ \ \textrm{ and }  \ Ê\   4 \gamma (\Lambda^\gamma_1 / \Lambda^\gamma_l +1) \leq \Lambda_l^\gamma .$ 

\item[2.]   Any choice of $\epsilon >0$ is supposed to be small enough so that:

$2 \epsilon < \gamma \ \ , \ \  4\epsilon + 2\gamma  < \Lambda_l \ \ \textrm{ and }  \ Ê\    \epsi  ( [\Lambda^\gamma_1 / \Lambda^\gamma_l] + 3 )  <  b . $
\end{itemize}
\end{note}

\subsection{Preparatory diagram along a negative orbit}\label{tiro}

Let $\gamma , \epsilon >0$ satisfying Constraints \ref{nota}. Let $\rho_\epsi$ be the $\epsi$-slow function provided by (\ref{doublev}) in Section \ref{OR}. Let $\hat x \in X$.  For every $n \in \Z$ one sets:
$$ \rho_n := \rho_\epsilon (\tau^n(\hat x))    \ \ , \ \    W_n := W_{\tau^n(\hat x)} .$$
The sequence $(\rho_n)_n$ is $\epsi$-slow and according to (\ref{doublev}) we have: 
$$W_n :  \D^k ( \rho_n) \to \D^k(\rho_{n+1}) \ \ , \ \ \Lip W_n \leq \theta := e^{-\Lambda_l + 2\epsilon} . $$
Let $(r_n)_n$ be any $\epsi$-slow sequence  such that $(r_n)_n \leq (\rho_n)_n$. We set 
$$ r_{\hat x} := r_0  / h_{\epsi} (\hat x)  \ \ , \ \  r_n^\gamma := r_n e^{-n \gamma} \ \ , \ \ n_{\hat x} := \min \{ n \geq 0 \ , \ e^{n(-\Lambda_l + 4\epsi + \gamma)} h_\epsi({\hat x}) \leq 1 \} . $$
The integer $n_{\hat x}$ is well defined since $4\epsi + \gamma < \Lambda_l$ by our Constraints  \ref{nota}.
We have defined  the charts $\psi_{\hat x} : E_{\hat x} \to \Pj^k$ in Section \ref{affchart}.

 \begin{lem}\label{tiroir} 
The following  diagram  commutes for every $n \geq n_{\hat x}$:
\[  \xymatrix{
  B_{x_0}(r_{\hat x}) \ar[rrr]^{ f^{-n}_{\hat x}  } \ar[d]_{C_{\epsi, \hat x } \circ \psi_{\hat x}^{-1}}& & & 
   B_{x_{-n}}(r_{\hat x} e^{-n(2\epsi + \gamma)}  ) \ar[d]^{C_{\epsi, \tau^n(\hat x)}  \circ \psi_{\tau^n(\hat x)}^{-1} }     \\
      \D^k(r_0)  \ar[rrr]^{ W_{n-1} \circ \cdots \circ W_0}  & &  & \D^k (r_n^\gamma ) .   }   \]
\end{lem}

 \begin{proof}
 The commutation follows from our previous definitions so we only have to check that each arrow is well defined.
We recall that, to  simplify,  we do not take into account the distortion induced by the charts $\psi_{\hat x}$. 

The left vertical arrow is well defined since $r_{\hat x} = r_0  / h_{\epsi} (\hat x)$ and $\Lip C_{\epsi, \hat x}  \leq  h_\epsi(\hat x)$. Similarly, to see that the right vertical arrow is well defined too, we observe that
$$ r_{\hat x}  e^{-n(2\epsi + \gamma)} h_\epsi(\tau^n(\hat x))  = r_0  e^{-n(2\epsi + \gamma)} h_\epsi(\tau^n(\hat x)) / h_\epsi(\hat x) \leq  r_0  e^{-n(\epsi + \gamma)} \leq r_n e^{-n \gamma} = r_n^\gamma . $$ 
Now, by Constraints \ref{nota} one has :
$$ r_0 \theta^n =r_0 e^{n(-\Lambda_l + 2\epsi)}  \leq r_0 e^{-n (\epsi + \gamma)} \leq r_n e^{-n\gamma} = r_n^\gamma $$
which, since $\Lip W_{n-1}\circ \cdots\circ W_0 \leq \theta^n$,
 shows that the bottom horizontal arrow is well defined.
Finally, one sees that the top horizontal arrow is well defined  since, according to the definition
of  $n_{\hat x}$, the map $\Phi_n :=  (C_{\epsi, \tau^n(\hat x)}  \circ \psi^{-1}_{\tau^n(\hat x)}) ^{-1}  \circ W_{n-1} \circ \cdots \circ W_0 \circ  (C_{\epsi, \hat x} \circ \psi_{\hat x}^{-1} ) $ satisfies
$$\Lip ( \Phi_n ) \leq e^{n(-\Lambda_l + 2\epsi)} h_\epsi (\hat x)   \leq e^{-n(2\epsi + \gamma)}$$ 
on the ball $B_{x_0}(r_{\hat x})$ for every $n \geq n_{\hat x}$. \finsec
\end{proof}

\subsection{Theorem A from the linearization of chains}\label{contrholo}

 Theorem A follows from Theorem B below, whose proof will be given in Section \ref{proofTHMB}. 

\begin{theorem}{B} \label{linn}
Let $\Lambda_l < \cdots < \Lambda_1$ be positive real numbers. Let $\gamma , \epsilon > 0$ satisfying  Constraints \ref{nota}. Let $(\rho_n)_{n \in \Z}$ be a  $\epsilon$-slow sequence and let  
$$W_n :  \D^k ( \rho_n) \to \D^k(\rho_{n+1}) $$
be a sequence of holomorphic contractions fixing the origin in $\C^k$ and such that $$\forall n \in \Z \ , \ \Lip W_n \leq \theta := e^{-\Lambda_l + 2 \epsi} . $$
We assume the existence of a  decomposition $\C^k = \oplus_{j=1}^l L_j$, where $L_j = \{ 0 \} \times \cdots \times \C^{k_j}  \times \cdots \times \{0 \}$ such that, for every $n \in \Z$,  the linear map  $A_n := d_0 W_n$ 
satisfies
   \[ A_n(L_j) = L_j \ , \  
 e^{-\Lambda_j - \epsilon} \abs{v} \leq \abs{A_n(v)}  \leq e^{-\Lambda_j + \epsilon}  \abs{v}  \ \ \textrm{for every}  \ v \in L_j . \]
Then there exists a  $\epsi$-slow sequence  $(r_n)_n \leq (\rho_n)_n$ and a sequence of holomorphic maps $$\varphi_n :  \D^k ( r^\gamma_n)  \to  \D^k(r_n) \ \ \textrm{where } \ \  r_n^\gamma := r_n  e^{-n\gamma} $$ 
such that

\begin{enumerate}
\item  $ e^{n(\gamma - 2\epsi)} \vert u - v \vert     \leq     \vert \varphi_n(u) - \varphi_n(v)  \vert     \leq  e^{n (\gamma + 2\epsi)}  \vert u - v \vert$, 
\item the following diagram commutes:
  \[  \xymatrix{
 \cdots  \ar[rr]^{ W_{n-1} } & &   \D^k ( r^\gamma_n )  \ar[rr]^{ W_n }  \ar[d]_{\varphi_n}  & &    \D^k(r^\gamma_{n+1})   \ar[rr]^{ W_{n+1} }  \ar[d]_{\varphi_{n+1}}  & &\cdots  \\
  \cdots \ar[rr]^{ A^\gamma_{n-1} } & &  \D^k( 4 r_n)  \ar[rr]^{ A^\gamma_n }  & &     \D^k( 4 r_{n+1})  \ar[rr] ^{ A^\gamma_{n+1} }        & &\cdots  }   \]
  \end{enumerate}
  where the maps $A_n^\gamma$ are given by $A_n^\gamma:=e^\gamma A_n$.
\end{theorem}

 Theorem A is simply obtained by combining  Lemma \ref{tiroir} with Theorem B and setting :
$$   \rho_{\hat x} = 4 r_0 \ , \ \varphi_{\hat x , n} = \varphi_n \circ C_{\epsi, \tau^n(\hat x)} \circ \psi_{\tau^n(\hat x)}^{-1}  \ , \  D_{\hat x , n} = A_{n-1}^\gamma \circ \cdots \circ A_0^\gamma. $$

Let us stress that the perturbation $A_n^\gamma = e^\gamma A_n$ precisely aims to shift the Lyapunov spectrum of  $(A_n)_n$ and cancel the resonances.

\begin{rem}
The statement of Theorem A only takes into account the right hand side part of the commutative diagram of Theorem B, which corresponds to the negative coordinates (the past) of $\hat x$. The fact that  Theorem B concerns sequences of  mappings $(W_n)_n$ indexed by $\Z$ is crucial. Indeed, we shall see in the proof of Lemma \ref{iter} that the construction of the change of coordinates  $\varphi_{\hat x,n}$ actually requires the positive and negative coordinates of $\hat x$.
\end{rem}

\subsection{Proof of Corollary}

Let $y$ be either $p$ or $q$. One sees on the commuting diagram in Theorem A that $f^n(y) \in B_{x_0}(tr'_{\hat x})$ and that $D_{\hat x,n}\circ \varphi_{\hat x,0} \left(f^n(y)\right) =\varphi_{\hat x,n}(y)$. Taking into account the Lipschitz  estimates on $\varphi_{\hat x,0}$ it follows that
\begin{eqnarray}\label{C1}
\varphi_{\hat x,n}(p) \, , \,  \varphi_{\hat x,n}(q) \in D_{\hat x,n} ( \D^k(tr_{\hat x})).
\end{eqnarray}
Let us also check that
\begin{eqnarray}\label{C2}
 D_{\hat x,n} ( \D^k(tr_{\hat x}))\subset \varphi_{\hat x,n} \left( B_{x_{-n}}(tr_{\hat x} e^{-n(\gamma+2\epsilon)})\right).
\end{eqnarray}
To see this we simply observe that
\begin{eqnarray*}
D_{\hat x,n}\left( \D^k(tr_{\hat x})\right)\subset
\D^k\left(tr_{\hat x} e^{-n(\Lambda_l -\gamma -\epsilon)}\right) \subset
\D^k\left(tr_{\hat x} e^{-4n\epsilon}\right) \subset
\varphi_{\hat x,n} \left( B_{x_{-n}}(tr_{\hat x} e^{-n(\gamma+2\epsilon)})   \right)
\end{eqnarray*}
where the second inclusion comes from $\Lambda_l > 2\gamma + 4\epsilon>\gamma +5\epsilon$ (see the Constraints \ref{nota})
and the last one from the Lipschitz estimate on $\varphi_{\hat x,n}$.

Now, by (\ref{C1}) there exists a segment $\Gamma$ connecting the two points  $\varphi_{\hat x,n}(p), \varphi_{\hat x,n}(q)$ within the \emph{convex set} $D_{\hat x,n}\left( \D^k(tr_{\hat x})\right)$. By the Lipschitz estimate
on $\varphi_{\hat x,n}$, this segment $\Gamma$ satisfies 
$$\textrm{length}(\Gamma)\le e^{n(\gamma+3\epsilon)}h_{\hat x} d(p,q) .$$
By (\ref{C2}), the image $\widetilde{\Gamma}$ of $\Gamma$ by the map $\left(\varphi_{\hat x,n}\right)^{-1}$ is a well defined smooth path connecting
$p$ and $q$ and contained in $\varphi_{\hat x,n}^{-1}\circ D_{\hat x,n} (\D^k(tr_{\hat x}))$. Again by the Lipschitz estimate on $\varphi_{\hat x,n}$ we get 
$$ \textrm{length}(\widetilde{\Gamma}) \le e^{n(-\gamma+2\epsilon)} \textrm{length}(\Gamma)\le e^{5n\epsilon}h_{\hat x} d(p,q) . $$
 Finally, 
$\widetilde{\Gamma} \subset 
\varphi_{\hat x,n}^{-1}\circ D_{\hat x,n} \circ \varphi_{\hat x ,0} \left(B_{x_0}(tr_{\hat x})\right)={\widetilde E}_{\hat x}^{-n}(t)$ by the Lipschitz estimate
on $\varphi_{\hat x,0}$.

\section{Linearization of chains of holomorphic contractions}\label{proofTHMB}

This Section is devoted to the proof of Theorem B.

\subsection*{Step 1 : Shifting the spectrum}
 
\begin{lem} \label{decal}
Let $W_n :  \D^k (\rho_n) \to \D^k(\rho_{n+1})$ be a sequence of holomorphic contractions satisfying the assumptions of Theorem B.
Let
$$ \rho^\gamma_n := \rho_n e^{-n \gamma} \ \ , \ \ \Delta_n := e^{n\gamma} \Id_{\C^k} \ \ , \ \ W^\gamma_n := \Delta_{n+1} \circ W_n \circ \Delta_n^{-1} . $$
Then
\begin{itemize}
\item the following diagram commutes
\[  \xymatrix{
 \cdots  \ar[rr]^{ W_{n-1} } & &   \D^k ( \rho^\gamma_n)  \ar[rr]^{ W_n }  \ar[d]_{\Delta_n}  & &    \D^k(\rho^\gamma_{n+1})   \ar[rr]^{ W_{n+1} }  \ar[d]_{\Delta_{n+1}}  & &\cdots  \\
  \cdots \ar[rr]^{ W^\gamma_{n-1} } & &  \D^k(\rho_n)  \ar[rr]^{ W^\gamma_n }  & &     \D^k(\rho_{n+1})  \ar[rr] ^{ W^\gamma_{n+1} }        & &\cdots  }   \]
\item $\Lip W_n^\gamma = \theta e^\gamma  <1$ and the linear part $A_n^\gamma  := e^\gamma A_n$ of  $W_n^\gamma$ satisfies
\[  e^{-\Lambda_1 - \epsilon + \gamma} \abs{v} \leq \abs{A^\gamma_n(v)}  \leq e^{-\Lambda_l + \epsilon + \gamma}  \abs{v}  \ \ \textrm{for every}  \ v \in \C^k . \]
\end{itemize}
\end{lem}

\begin{proof} Recall that according to Constraints \ref{nota} we have $\theta e^{2 \gamma} = e^{-\Lambda_l +2\epsilon +2 \gamma} < 1$ and  $(\rho_n^\gamma)_{n\in \Z}$ is $2\gamma$-slow since $\epsilon < \gamma$. In particular 
$\Lip W_n^\gamma = e^\gamma \Lip W_n = e^\gamma  \theta<1$.
It remains to check that $W_n : \D^k(\rho^\gamma_n) \to \D^k(\rho^\gamma_{n+1})$, which is clear since
$\rho^\gamma_n \cdot \Lip W_n  \le \rho^\gamma_n \theta \leq \rho^\gamma_{n+1} \theta e^{2\gamma} <\rho^\gamma_{n+1}$.\fin
\end{proof}

\subsection*{Step 2 : Improving the order of contact with the linear part}

\begin{prop}\label{thm1}
Let $W_n :  \D^k (\rho_n) \to \D^k(\rho_{n+1})$ be a sequence of holomorphic contractions satisfying the assumptions of Theorem B.  Let $W^\gamma_n : \D^k ( \rho_n) \to \D^k(\rho_{n+1})$ be the sequence of holomorphic contractions given by Lemma \ref{decal}. There exist a $\epsi$-slow sequence  $(r_n)_n \leq (\rho_n)_n$ and a sequence of holomorphic maps $$T^1_n :  \D^k (r_n)  \to  \D^k (2r_n)$$ such that $d_0 T_n^1  = \Id_{\C^k}$ and
\begin{enumerate}
\item $\forall n \in \Z$, $\forall (u,v) \in \D^k(r_n) \times \D^k(r_n)$, $e^{-\epsi} \abs{u-v} \leq
  \abs{T_n^1 (u) - T_n^1 (v)} \leq e^\epsi \abs{u-v}$ ,
\item the following diagram commutes
\[  \xymatrix{
 \cdots  \ar[rr]^{ W^\gamma_{n-1} } & &   \D^k (r_n)  \ar[rr]^{ W^\gamma_n }  \ar[d]_{T^1_n}  & &    \D^k( r_{n+1})   \ar[rr]^{ W^\gamma_{n+1} }  \ar[d]_{T^1_{n+1}}  & &\cdots  \\
  \cdots \ar[rr]^{ X_{n-1} } & &  \D^k(2 r_n)  \ar[rr]^{ X_n }                  & &     \D^k(2 r_{n+1})  \ar[rr] ^{ X_{n+1} }        & &\cdots  }   \]
where $X_n = A^\gamma_n  + O([\Lambda^\gamma_1 / \Lambda^\gamma_l] +2)$ for every $n \in \Z$.
\end{enumerate}
\end{prop}

The proof consists in applying a finite number of times Lemma \ref{iter} below. Let us say for convenience that  a  sequence 
$(G_n)_{n \in \Z}$ of 
holomorphic mappings $G_n : \D^k(a_n) \to \D^k(a_{n+1})$
 is $\epsi$-\emph{slow} if  $(a_n)_{n \in \Z}$ is a $\epsilon$-slow sequence. 
Assume that the linear part of such a sequence $(G_n)_{n \in \Z}$ is equal to $(A^\gamma_n)_{n \in \Z}$. Then, according to  our discussion in Subsection \ref{RC}, the $\Lambda_j^\gamma$'s do not satisfy any resonance relation and thus, given any integer $p \geq 2$, Lemma \ref{iter} allows to conjugate $(G_n)_{n \in \Z}$ to some new 
$\epsi$-\emph{slow} sequence whose linear part is still equal to $(A^\gamma_n)_{n \in \Z}$ but with no homogeneous part
of degree $p$ in its Taylor expansion. The cancellation of the $p$-homogeneous part of $(G_n)_{n \in \Z}$ relies on the resolution 
of  some so-called homological equations. We shall only sketch the proofs,  the details are in 
\cite[Subsection 3.2]{BDM} or \cite{JV}. 

\begin{lem}\label{iter} Let  $\left(G_n \right)_{n\in \Z}$ be a  $\epsi$-slow sequence of holomorphic contractions fixing the origin and whose linear part is equal to  $(A^\gamma_n)_{n \in \Z}$. Let $p \geq 2$ and let  $G_n^{(p)}$  be the $p$-homogeneous part of the  Taylor expansion of $G_n$.
Then:
\begin{enumerate}
\item there exists a sequence of $p$-homogeneous polynomial maps  $(H_n)_{n \in \Z}$ such that
\[ \forall n \in \Z \ \ , \ \ G_n^{(p)} +  H_{n+1} \circ A^\gamma_n - A^\gamma_n \circ H_n  = 0 , \]
\item for $S^{(p)}_n := \Id + H_n$ and $\epsi' > 0$ 
there exists a $\epsi$-slow sequence $(\tau_n)_{n \in \Z}$ such that
$$ \forall n \in \Z \ \ , \ \ \forall (u,v) \in \D^k(\tau_n) \times \D^k(\tau_n) \ \ , \ \ e^{-\epsi '} \abs{u-v} \leq
  \abs{S^{(p)}_n (u) - S^{(p)}_n (v)} \leq e^{\epsi '} \abs{u-v} $$
  and for which the following diagram commutes
\[  \xymatrix{
 \cdots  \ar[rr]^{ G_{n-1} } & &   \D^k (\tau_n)  \ar[rr]^{ G_n }  \ar[d]_{S^{(p)}_n}  & &    \D^k(\tau_{n+1})   \ar[rr]^{ G_{n+1} }  \ar[d]_{S^{(p)}_{n+1}}  & &\cdots  \\
  \cdots \ar[rr]^{ G^{\widehat p}_{n-1} } & &  \D^k(e^\epsi \tau_n)  \ar[rr]^{ G^{\widehat p}_n }                  & &     \D^k(e^\epsi \tau_{n+1})  \ar[rr] ^{ G^{\widehat p}_{n+1} }        & &\cdots  }   \]
where  $G^{\widehat p}_n = A_n^\gamma +  G_n^{(2)} + \cdots + G_n^{(p-1)}  + 0 + O(p+1)$.
\end{enumerate}
 \end{lem}
 
 \subsection*{Proof of Proposition \ref{thm1}}

Let $p^* := [\Lambda^\gamma_1 / \Lambda^\gamma_l] + 1$ and $\epsilon ' := \epsilon / p^*$. We successively apply  Lemma \ref{iter} with
 $$p = 2 \ , \ G_n = W_n^\gamma \ ,$$
and then with
$$ p=3 \ , \ G_n = (W_n^{\gamma})^{\widehat 2} \  , $$
and so on, up to 
$$p = p^*  \ , \ G_n = (\cdots ((((W^\gamma_n)^{\widehat {2}})^{\widehat 3})^{\widehat 4}) \cdots ) ^{\widehat{p^*-1}} \ . $$
This yieds a $\epsi$-slow sequence $(r_n)_n \leq (\rho_n)_n$ such that the following diagram commutes
\[  \xymatrix{
 \cdots  \ar[rr]^{ W^\gamma_{n-1} } & &   \D^k (r_n)  \ar[rr]^{ W^\gamma_n }  \ar[d]_{T^1_n}  & &    \D^k( r_{n+1})   \ar[rr]^{ W^\gamma_{n+1} }  \ar[d]_{T^1_{n+1}}  & &\cdots  \\
  \cdots \ar[rr]^{ X_{n-1} } & &  \D^k(e^{p^* \epsi} r_n)  \ar[rr]^{ X_n }                  & &     \D^k(e^{p^* \epsi} r_{n+1})  \ar[rr] ^{ X_{n+1} }        & &\cdots  }   \]
where $$X_n :=  (\cdots ((((W^\gamma_n)^{\widehat {2}})^{\widehat 3})^{\widehat 4}) \cdots ) ^{\widehat{p^*}} = A^\gamma_n  + O([\Lambda^\gamma_1 / \Lambda^\gamma_l] +2)$$
and $$ T^1_n := S^{(p^*)}_n \circ \cdots \circ S^{(2)}_n . $$
Since  $p^* \epsi < \frac{1}{2} \ln 4$ by Constraints \ref{nota}, we have  $T^1_n : \D^k(r_n) \to \D^k(2 r_n)$. Since $p^* \epsi' = \epsilon$, we also have  $e^{-\epsi} \abs{u-v} \leq \abs{T_n^1 (u) - T_n^1 (v)} \leq e^\epsi \abs{u-v}$ on $\D^k(r_n)$. \finsec

 \subsection*{Proof of Lemma \ref{iter}}
 
To prove Item 1 amounts to show that the following linear operator  $\Gamma$, defined on the space of slow sequences of $p$-homogeneous mappings, is surjective:
\[   \Gamma :  (H_n)_{n \in \Z}   \mapsto (H_{n+1} \circ A^\gamma_n - A^\gamma_n \circ H_n)_{n \in \Z} .  \]
To simplify we suppose that every exponent $\Lambda_j$ has multiplicity $k_j = 1$ (therefore $l=k$ and $\lambda_j = \Lambda_j$ for $1\le j\le k$). Let us fix $$ \al \in \N^k \textrm{ such that }  2 \leq \vert \alpha \vert \leq [\Lambda^\gamma_1 / \Lambda^\gamma_l] + 1.$$
We are going to exhibit a preimage for $\Gamma$ to the sequence 
$$G_n^{(p)} =  (0,\cdots, a_n z^\al , \cdots , 0) , $$
where $a_n z^\al$ is at the $j$-th place. 

For this purpose, we shall use the  control 
$\al_1 \lambda^\gamma_1 +\cdots +\al_k \lambda^\gamma_k - \Lambda^\gamma_j  \notin [-b , b] $
(see Subsection \ref{RC}) and 
 treat separately the  cases $\al_1 \lambda^\gamma_1 +\cdots +\al_k \lambda^\gamma_k - \Lambda^\gamma_j>b$ and $\al_1 \lambda^\gamma_1 +\cdots +\al_k \lambda^\gamma_k - \Lambda^\gamma_j< -b$. In the first case, we shall solve the equation $\Gamma ( (H_n)_{n \in \Z} ) = ( G_n ^{(p)})_{n \in \Z}$ by averaging on $n\geq 0$ (which correspond to the past of $\hat x$) and by averaging on  $n\leq 0$ (which correspond to the future of $\hat x$) in the second case.\\

$\bullet$ If $\al_1\lambda^\gamma_1 +\cdots +\al_k \lambda^\gamma_k -  \Lambda^\gamma_j > b$, we set 
\begin{equation}\label{plp} \forall n \in \Z \ \ , \ \ H_n := (A_n^{\gamma}) ^{-1} G_n^{(p)} + \sum_{r \geq 1} (A^\gamma_n) ^{-1} \cdots (A^\gamma_{n+r})^{-1} \ G_{n+r}^{(p)} \  A^\gamma_{n+r-1} \cdots A^\gamma_n . 
\end{equation}
A formal computation shows that $\Gamma ( (- H_n)_{n \in \Z} ) = ( G_n ^{(p)} )_{n \in \Z}$, which means:    
\[ \forall n \in \Z \ \ , \ \ G_n^{(p)} +  H_{n+1} \circ A^\gamma_n - A^\gamma_n \circ H_n  =  0  . \]

The convergence of the series (\ref{plp}) relies on the block diagonal property of $(A^\gamma_n)_{n \in \Z}$. To simplify, let us assume that the $(A^\gamma_n)$ are truly diagonal. Writing  $H_n$ as  $H_n = P_{n,0} + \sum_{r \geq 1} P_{n,r}$, we obtain 
\begin{displaymath}
\begin{array}{rcl}
 \abs{P_{n,r}(z)} &  \leq  &  \abs{ (A^\gamma_n) ^{-1} \, \cdots  \, (A^\gamma_{n+r}) ^{-1} \ G_{n+r}^{(p)} \, \left( \, e^{ - r \lambda^\gamma_1 + r \epsi} \, z_1 \, , \, \cdots \, , \, e^{- r \lambda^\gamma_k + r \epsi } \, z_k \, \right)  } \\
   &   \leq  &   \abs{ (A^\gamma_n) ^{-1} \, \cdots  \, (A^\gamma_{n+r}) ^{-1}  \ \left( 0 \, , \, \cdots \, , \, a_{n+r} \, e^{ - r (\, \al_1 \lambda^\gamma_1 \, + \, \cdots \, + \, \al_k  \lambda^\gamma_k \, ) + \abs{\al}  r \epsi } \, z^\al \, , \, \cdots \, , \, 0 \, \right)  }   \\
   &   \leq  &  e^{r \Lambda^\gamma_j + r \epsi} \,  e^{ - r ( \, \al_1 \lambda^\gamma_1 \, + \, \cdots \, + \, \al_k \lambda^\gamma_k \, ) + \abs{\al} r \epsi} \,  \abs{a_{n+r}} \,  \abs{z^\al} \\
   &   \leq  &  e^{ - r b}   e^{ r ( \abs{\al} + 1 ) \epsi } \,  \abs{a_{n+r}} \,  \abs{z^\al} \\
   &   \leq  &  e^{ - r b }  e^{ r ( [\Lambda^\gamma_1 / \Lambda^\gamma_l] + 3 ) \epsi } \,  \abs{a_n} \,  \abs{z^\al}.
\end{array}                                      
\end{displaymath} 
where the last inequality uses  $\vert \alpha \vert \leq [\Lambda^\gamma_1 / \Lambda^\gamma_l] + 1$ and the fact, easily checked by using  Cauchy's estimates, that  $(a_n)_{n \in \Z}$ is a $\epsi$-slow sequence. By Constraints \ref{nota}, $b  -  ( [\Lambda^\gamma_1 / \Lambda^\gamma_l] + 3 ) \epsi > 0$ and thus the series (\ref{plp}) converge.\\

$\bullet$ If $\al_1\lambda^\gamma_1 +\cdots +\al_k \lambda^\gamma_k -  \Lambda^\gamma_i < -b$, we proceed as before by setting 
\[ \forall n \in \Z \ \ , \ \ H_n := - G_{n-1}^{(p)} (A^\gamma_n)^{-1}  - \sum_{r \geq 1} A^\gamma_{n-1} \cdots A^\gamma_{n-r} \ G_{n-(r+1)}^{(p)} \ (A^\gamma_{n-r}) ^{-1} \cdots (A^\gamma_n)^{-1}  .\]

Let us now deal with Item 2 of Lemma \ref{iter}. We first set $S^{(p)}_n :=  \Id + H_n$ and observe that $(S^{(p)}_n)^{-1} = \Id - H_n + O(p+1)$ and \[ (S^{(p)}_{n+1})^{-1} \circ G_n \circ S^{(p)}_n    = \left( A^\gamma_n +  G_n^{(2)} + \cdots + G_n^{(p-1)}  \right) + \left(  G_n^{(p)} +  H_{n+1} \circ A^\gamma_n - A^\gamma_n \circ H_n  \right) + O(p+1). \]
The second term in the right hand side vanishes by construction of $(H_n)_n$. Finally, Lemma \ref{NNbis} on bundle maps ensures the existence of a $\epsi$-slow sequence $(\tau_n)_n$ satisfying the required properties.  \\

In order to prove the general case where $G_n^{(p)}$ is a $p$-homogeneous map but not a monomial map, we proceed by linearity in Item 1 to obtain:
\[ G_n^{(p)} +  \sum  H^{j,\alpha}_{n+1} \circ A^\gamma_n - A^\gamma_n \circ \sum  H^{j,\alpha}_n  =  0  , \]
where the sum ranges over $j \in \{ 1 , \cdots , l \}$ et $\vert \alpha \vert = p$.
Item 2 then follows by performing the change of coordinates $S^{(p)}_n :=  \Id +  \sum  H^{j,\alpha}_n$. \fin


\subsection*{Step 3: Conjugation to a linear mapping and conclusion}

Proposition \ref{iter2} below is a special case of Theorem 1.1 in \cite{BDM}, it shows that a slow sequence of holomorphic mappings $(X_n)_{n \in \Z}$ is conjugated to the sequence of its linear part  $(A^\gamma_n)_{n \in \Z}$ once these two sequences have a sufficiently large contact order. Let us stress that this step does not require the fact that  $(A_n^\gamma)_{n \in \Z}$ is block diagonal. 
 
 \begin{prop}\label{iter2} Let $(X_n)_{n \in \Z}$ be a $\epsi$-slow sequence of holomorphic mappings with linear parts  $(A^\gamma_n)_{n \in \Z}$. Assume that there exist $0 < m \leq M <1$ such that
\[ \forall n \in \Z \ \ , \ \  \forall v  \in \C^k \ \ , \ \ m \abs v \leq \abs{A^\gamma_n (v)} \leq M \abs v   \] 
and that
 \[  X_n = A^\gamma_n + O(q+1) \ \  \textrm{ where $q \geq 1$ satisfies  } \ \ ( M e^{2\epsi} ) ^{q+1} < m e^{-\epsilon} .  \] 
Then there exist a $\epsi$-slow sequence $(r_n)_{n\in \Z}$ and a sequence $\left(T_n^2\right)_{n\in Z}$ of holomorphic mappings $T_n^2 :  \D^k (r_n)  \to  \D^k (2r_n)$ such that $d_0 T_n^2  = \Id$ and 
\begin{enumerate}
\item $\forall n \in \Z$, $\forall (u,v) \in \D^k(r_n) \times \D^k(r_n)$, $e^{-\epsi} \abs{u-v} \leq
  \abs{T_n^2 (u) - T_n^2 (v)} \leq e^\epsi \abs{u-v}$ ,
\item the following diagram commutes:
\[  \xymatrix{
 \cdots  \ar[rr]^{ X_{n-1} } & &   \D^k (r_n)  \ar[rr]^{ X_n }  \ar[d]_{T^2_n}  & &    \D^k(r_{n+1})   \ar[rr]^{ X_{n+1} }  \ar[d]_{T^2_{n+1}}  & &\cdots  \\
  \cdots \ar[rr]^{ A^\gamma_{n-1} } & &  \D^k(2r_n)  \ar[rr]^{ A^\gamma_n }                  & &     \D^k(2r_{n+1})  \ar[rr] ^{ A^\gamma_{n+1} }        & &\cdots  }   \]
\end{enumerate}
\end{prop}

 \subsection*{Proof of Proposition \ref{iter2}}
To deduce the Proposition from Theorem 1.1 of \cite{BDM} we observe that the proof of this theorem starts by fixing a number $\theta$ such that $M^{q+1}/m<\theta<1$ and then chosing $\epsilon >0$ small enough so that $Me^{2\epsilon} \le 1$ and $e^{(q+3)\epsilon} M^{q+1}/m \le \theta$. Our assumption $(M e^{2\epsi} ) ^{q+1} < me^{-\epsilon}$ implies both of these conditions, with $\theta=e^{-q\epsilon}$ for the second one.

For the reader's convenience we now sketch the proof of Proposition \ref{iter2},  details can be found in  \cite[Subsection 3.1]{BDM}. 
Let $n \in \Z$ and let 
\[ \forall p \geq 0 \ ,  \ A^\gamma_{p,n} := A^\gamma_{n+p} \circ \cdots \circ A^\gamma_n \ \ , \ \ X_{p,n} := X_{n+p} \circ \cdots \circ X_n \]
with the convention $A^\gamma_{-1,n} = X_{-1,n} = \Id$. The germs of formal limits 
\[ T^2_n := \lim_{p \to + \infty} (A^\gamma_{p,n}) ^{-1} \circ X_{p,n} \] 
satisfy $d_0 T_n^2 = Id$ and give the commutative diagram of Proposition \ref{iter2}.  Let us show the convergence by induction. Let $\beta := (Me^{2\epsi}) ^{q+1} / m < 1$ and consider the assertion
\[ \PP(p) \ : \ \forall n \in \Z \ , \  \abs{ ((A^\gamma_{p,n}) ^{-1} \circ X_{p,n}  - (A^\gamma_{p-1,n}) ^{-1} \circ X_{p-1,n}) (v)  } \leq ( \beta^{p}  /   m r_n^{q+1}  )   \, \abs{v}^{q+1}  .   \]
Then $\PP(0)$ is satisfied: the property $X_n = A^\gamma_n + O(q+1)$ indeed shows that  
 \[  \forall n \in \Z \ , \  \abs{ ( (A^\gamma_n)^{-1} \circ X_n  - \Id) (v)  } = \abs{  (A^\gamma_n)^{-1} \circ ( X_n  - A^\gamma_n) (v)  }  \leq  ( 1  /   m r_n^{q+1}  ) \abs{v}^{q+1} .   \]
The real numbers $m$ and $r_n^{q+1}$ respectively come from the lower bound $m \abs v \leq \abs{A^\gamma_n (v)}$ and from an estimate of $(X_n  - A^\gamma_n) (v)$ using Cauchy's estimates. Now we show that $\PP(p)$ implies $\PP(p+1)$. Let us perform a  right composition of $\PP(p)$ by $X_n$ (we have also replaced $n$ by $n+1$):  
 \[ \forall n \in \Z \ , \   \abs{ ((A^\gamma_{p,n+1})^{-1} \circ X_{p,n+1}  - (A^\gamma_{p-1,n+1}) ^{-1} \circ X_{p-1,n+1} ) (X_n(v))  } \leq ( \beta^{p}  /   m r_{n+1}^{q+1}  ) \abs{X_n(v)}^{q+1} .    \]
Then, by using the identity $X_{p,n+1} \circ X_n = X_{p+1,n}$ and the comparison of $X_n$ to its linear part $A_n^\gamma$ given by Lemma \ref{NN}:
  \[ \forall n \in \Z \ , \   \abs{ ((A^\gamma_{p,n+1})^{-1} \circ X_{p+1,n }  - (A^\gamma_{p-1,n+1})^{-1} \circ X_{p,n} ) (v)  } \leq  ( \beta^{p}  /   m r_{n+1}^{q+1} ) (M e^\epsi)^{q+1} \abs{v}^{q+1} .    \]
A left composition by $(A^\gamma_n)^{-1}$ gives by using the fact that  $(r_n)_n$ is slow: 
 \[  \forall n \in \Z \ , \   \abs{ ((A^\gamma_{p+1,n})^{-1} \circ X_{p+1,n }  - (A^\gamma_{p,n})^{-1} \circ X_{p,n} ) (v)  } \leq ( \beta^{p}  /   m r_n^{q+1} ) \, {(M e^{2\epsi}) ^{q+1} \over m }  \abs{v}^{q+1} .   \]
This last quantity is  equal to $( \beta^{p+1}  /   m r_n^{q+1} )  \abs{v}^{q+1}$, which shows $\PP(p+1)$. Finally, it suffices to apply Lemma \ref{NNbis} to end the proof of Proposition \ref{iter2}.\fin \\

To complete the proof of  Theorem B, we successively apply Lemma \ref{decal}, Proposition \ref{iter} and Proposition \ref{iter2}  with 
$$  q = [\Lambda^\gamma_1 / \Lambda^\gamma_l] +1 \ \ , \ \ m = e^{-\Lambda^\gamma_1 - \epsilon + \gamma} \ \ , \ \ M  = e^{-\Lambda^\gamma_l + \epsilon + \gamma} . $$
We have $m \abs v \leq \abs{A^\gamma_n (v)} \leq M \abs v$ from  Lemma \ref{decal} and $(M e^{2\epsi})^{q+1} < m e^{-\epsilon}$ from Lemma  \ref{donne} (a consequence of Constraints \ref{nota}). Theorem B follows by setting $$\varphi_n := T^2_n \circ T^1_n \circ \Delta_n .$$
This completes the proof of Proposition \ref{iter2}.

\section{Applications}

We review here some results whose proofs crucially rely on our Distortion Theorem. Proving similar results in the one-dimensional setting  requires the classical Koebe distortion Theorem.

\subsection{Multipliers of repelling cycles}\label{mult}

The following result, first proved in \cite{BDM}, plays an important role in the study of bifurcations within holomorphic families of endomorphisms (see \cite{Be} and \cite{BBD}).

\begin{thm}\label{replyap}
Let $f : \Pj^k \to \Pj^k$ be a holomorphic endomorphism of degree $d \geq 2$.  Let $\lambda_k \leq \cdots \leq \lambda_1$ be the Lyapunov exponents of its equilibrium measure. Then 
\[ \lim_{n \to + \infty}  {1 \over d^{kn}} \sum_{p \in \RR_n}  \log \Jac f(p)  = (\lambda_1 + \cdots + \lambda_k)  \]
where $\RR_n$ is the set  of $n$-periodic repelling points of $f$.
\end{thm}

\textsc{Sketch of proof:}
 To start with, we reprove that  the equilibrium measure $\mu$ of $f$ equidistributes the repelling cycles of $f$
 (this is a theorem due to Lyubich \cite{L} for $k=1$ and Briend-Duval \cite{BD} for $k\ge 1$) we follow here Briend-Duval approach. Let $B : = B_x(r)$ be a small ball around  a $\mu$-generic point $x$. Since $\mu$ is mixing, we have $\mu(f^{-n}B \cap B) \simeq \mu(B)^2$ for $n$ large enough. Now let $\FF_n(B)$ be the set of inverse branches $g_n$  of $f^n$ defined on $B$ and with image in $B$. By using $f^*\mu = d^k \mu$ and the fact that the inverse branches are pairwise disjoint, the mixing property gives $\cdb \FF_n(B)  \cdot\mu(B) / d^{kn} \simeq \mu (B)^2 $, therefore:
\[   {1  \over d^{kn}} \cdb \FF_n(B)   \simeq \mu (B) . \] 
Since the Lyapunov exponents of  $\mu$ are positive, every element $g_n$ of $\FF_n(B)$ is a contracting map from $B$ to $B$, hence produces a repelling point $p$ for the iterate $f^n$. This implies
\[  {1 \over d^{kn}} \cdb \RR_n \cap B \geq \mu(B) . \]
Thus every cluster measure $\mu'$ of  ${1 \over d^{kn}} \sum_{p \in \RR_n} \delta_p$ satisfies $\mu' \geq \mu$. This implies $\mu' = \mu$  since the number of $n$-periodic point of $f$ is bounded above by $d^{kn}$, see \cite[Proposition 1.3]{DS}.\\

To obtain Theorem \ref{replyap}, we combine the Distortion Theorem with the above arguments to get:
\begin{equation}\label{gqd}
\forall g_n \in \FF_n(B) \ , \ \forall p \in g_n(B) \ , \ d_p f^n \simeq d_{g_n(x)} f^n .  
\end{equation}
Since $x$ is $\mu$-generic, we deduce from (\ref{gqd}) and from the definition of the Lyapunov exponents:
 \begin{equation*}\label{estsi}
 \  {1 \over n} \log \vert  \Jac f^n(p)\vert  \ \simeq \  (\lambda_1 + \cdots + \lambda_k). 
\end{equation*}
Let us specify that to make the approximations $\simeq$  precise, one has to work with the natural extension of $f$ and use the estimate concerning the change of coordinates $\varphi_{\hat x , n}$.

\subsection{Dimension of measures}\label{minor}

Let $f : \Pj^k \to \Pj^k$ be a holomorphic endomorphism and $\nu$ be an invariant measure.
One defines the pointwise dimensions $\underline{d}_{\nu}$ and $\bar{d_{\nu}}$ by
$$\forall x \in \Pj^k \ , \ \underline{d}_{\nu}(x)  := \liminf_{r \to 0} \, { \log \nu (B_x(r)) \over  \log r } \ \ , \ \ \bar{d_{\nu}}(x) :=  \limsup_{r \to 0} \, { \log \nu (B_x(r)) \over  \log r } .$$
When $\nu$ is ergodic, these functions are  $\nu$-almost everywhere constant and their generic values are denoted  $\underline{d}_{\nu}$ and $\bar{d_{\nu}}$. Young (\cite{Y}) proved that if $ a \leq \underline{d}_{\nu} \leq \bar{d_{\nu}} \leq b$, then
$$a \leq \dim_H(\nu) \leq b,$$ 
where $\dim_H(\nu) := \inf \{\dim_H(A), A \textrm{ Borel set } , \nu(A) = 1 \}$ is the Hausdorff dimension of $\nu$.\\

For the equilibrium measure $\mu$ of any degree $d\ge 2$ holomorphic endomorphism on $\Pj^k$,
it has been conjectured by Binder and DeMarco \cite{BDeM} that
\[  \dim_H (\mu)  = { \log d \over \lambda_1 } + \cdots + { \log d \over \lambda_k } \] 
where $\lambda_k \leq \cdots \leq \lambda_1$ are the Lyapunov exponents of $\mu$.

When $k=1$, this conjecture corresponds to a result of Ma\~n\'e \cite{M} who actually proved  that $\underline d_{\mu} = \bar d_\mu = \log d / \lambda$.  Our Distortion Theorem allows to obtain the following lower bounds in any dimension. 

\begin{thm}[Dupont \cite{D1}]\label{mino} Let $f : \Pj^k \to \Pj^k$ be a holomorphic endomorphism of degree $d \geq 2$. Let $\nu$ be an ergodic measure with positive Lyapunov exponents $\lambda_k \leq \cdots \leq \lambda_1$ and  entropy $h_\nu$. Then
$$ \underline{d}_{\nu} \geq {\log d^{k-1} \over \lambda_1} + { h_\nu -  \log d^{k-1} \over \lambda_k }  .$$
\end{thm}
The following corollary yields the lower bound of the above Conjecture in dimension $k=2$.
\begin{cor}
Let $f : \Pj^2 \to \Pj^2$ be a holomorphic endomorphism of degree $d \geq 2$. Let $\mu$ be its equilibrium measure and let $\lambda_1 \geq \lambda_2$ be the Lyapunov exponents of $\mu$. Then
$$   \underline d_\mu  \geq  { \log d \over \lambda_1 } +  { \log d \over \lambda_2 }. $$
\end{cor}

The proof consists in studying the distribution of inverse branches of $f^n$ in $\Pj^k$ and, in particular, uses an area growth argument which requires a precise description of the geometry of these  branches. This is where the  Distortion Theorem enters into the picture.\\ 

\textsc{Sketch of proof:}
 We first establish an upper bound for the cardinality of inverse branches of $f^n$ in generic balls of radius $e^{-n \lambda_k}$. Let $q_n$ be the entire part of $n \lambda_k / \lambda_1$ and let $B_x^{A}(r) := B_x(r) \cap A$.\\

{\bf Fact }{ \it  For every $\epsi >0$, there exist $\Omega_\epsi
  \subset \Pj^k$ and $r  > 0$ satisfying $\nu(\Omega_\epsi) > 1 - \epsi$ and the following property. Let $\EE_{r} \subset \Pj^k$ be a maximal $r$-separated subset. Then for every $x \in  \Omega_\epsi$ and $n$ large enough, the collection of inverse branches
   \[ \PP_n(x) := \left \{  \, f^{-n}_{\hat y_n} B_p (r) \ , \ y \in  B_x^{\Omega_\epsi} (r e^{-n
  \lambda_k  } )  \ , \ p \in \EE_r  \ , \ d(p,y_n) < r  \,  \right \}    \]
 is well defined and satisfies $\cdb \PP_n(x) \leq  d^{(k-1)(n - q_n)} e^{n\epsi}$.
}\\

We outline the proof of this Fact in three steps. It relies on Area estimates and on our Distortion Theorem. 
 Let $\om$ denotes the Fubini-Study form on $\Pj^k$ and $\eta : \D^{k-1} (2) \to \Pj^k$ be a holomorphic polydisc. The first step is to show that 
\begin{eqnarray} \label{a} \forall m \geq 1 \ , \ \aire \,  f^{m} \circ \eta_{\vert \D^{k-1}} := \int_{\D^{k-1}} (f^m \circ \eta)^* \om^{k-1}  \leq d^{(k-1) m} .
\end{eqnarray}
This is obtained by replacing $\om$ by the Green current $T$ of $f$ (using cohomologous arguments and integration by parts) and then using $f^* T = d T$, see \cite{Di, D1} for more details.

For the second step, let $x \in
\Omega_\epsi$ and denote by  $\LL_n$  the set of polydiscs 
 $L_n :  \D^{k-1} \to B_x (e^{-n \lambda_k})$. Applying (\ref{a}) to $m = n - q_n$ and 
$\eta = f^{q_n} \circ L_n$ yields
\begin{eqnarray}\label{b}
 \aire \, f^n \circ L_n  = \aire  f^{n-q_n} ( f^{q_n} \circ L_n )  \leq d^{(k-1)(n-q_n)} \;\textrm{for every}\; L_n \in \LL_n.
\end{eqnarray}
Note that $q_n$ is chosen to produce polydiscs of uniformly bounded sizes. Indeed, since $\lambda_1$ is the largest exponent and since $q_n \lambda_1  \simeq n \lambda_k$, we have  $$  f^{q_n} (L_n ) \subset f^{q_n} ( B_x
(e^{-n \lambda_k}) ) \subset B_{f^{q_n}(x)} (e^{-n \lambda_k} \cdot
e^{q_n \lambda_1}) \simeq B_{f^{q_n}(x)}(1) .$$ 

The third step is to prove  the existence of $\FF_n \subset
  \LL_n$ whose cardinality is at most $e^{n\epsi}$ and such that
 for every $P_n \in \PP_n(x)$ there exists a $(k-1)$-polydisc $L_n \in \FF_n$ satisfying 
 \begin{eqnarray}\label{c}
    \aire \, f^n \circ  {L_n}_{\vert L_n^{-1}(P_n)}  \geq 1 . 
   \end{eqnarray}
This relies on the geometric description of inverse branches given by our Distortion Theorem: every $P_n \in \PP_n(x)$ is indeed a parallelepiped with  dimensions $e^{-n \lambda_1} \leq \cdots \leq e^{-n \lambda_k}$. Precisely, the polydisc $L_n$ is transverse to the $e^{-n\lambda_k}$-direction of $P_n$ and $\FF_n$ is a collection of hyperplanes parallel to the coordinates axis.

The Fact finally comes from (\ref{b}) and (\ref{c}) since the inverse branches are pairwise disjoint: this gives $\cdb \PP_n(x) \leq  d^{(k-1)(n - q_n)} e^{n\epsi}$ as desired. \\

Let us now explain how the Fact implies Theorem \ref{mino}. We shall ignore the $e^{\pm n \epsilon}$ error terms due to the non-uniform hyperbolic setting. Let $x \in \Omega_\epsi$ and  $\rho_n :=  e^{-n \lambda_k }$. Since $\PP_n(x)$ covers the generic ball $B_x^{\Omega_\epsi}(\rho_n)$ (Fact) and since $\nu ( P ) \simeq e^{-n h_\nu}$ for every $P \in \PP_n(x)$ (Brin-Katok Theorem), we get 
\[  \nu (B_x^{\Omega_\epsi}(\rho_n) )  \leq  \sum_{P \in \PP_n(x)} \nu( P ) \leq  \, \cdb
\PP_n (x) \cdot e^{-n h_\nu }.  \]  
The Density Theorem then implies for $\nu$ almost every $x \in \Omega_\epsi$ and for $n$ large enough: 
\[  \nu \left( B_x(\rho_n) \right) \leq 2 \, \cdb \PP_n (x) \cdot e^{-n h_\nu }. \]
The upper bound on $\cdb \PP_n(x)$ given by the Fact yields:
\begin{equation*} \label{dff}
 \liminf_{r \to 0} \, { \log \nu (B_x(r)) \over  \log r } \, \geq \, 
   { \log d^{k-1} \over \lambda_1 }  + {h_\nu - \log d^{k-1} \over
   \lambda_k}  .
\end{equation*}
This estimate occurs $\nu$-almost everywhere since $\Omega :=\limsup \Omega_{1/q}$ has full $\nu$-measure.

\subsection{Dimension of currents}\label{currents}

In this Subsection we study the thickness of currents $S$ supporting dilating ergodic measures. Precisely, we focus on  a lower bound on their local upper pointwise dimension.
Let $S$ be a positive closed current of bidegree $(1,1)$ on $\Pj^2$ and let
$$\forall x \in \Pj^2 \ , \ \bar{d_S}(x)  := \limsup_{r \to 0} \frac{\log S \wedge \omega (B_x(r))}{\log r} . $$
We have $\bar{d_S}(x) \geq 2$, see \cite[Chapitre 3, \S5]{D}. For every $\Lambda \subset \Pj^2$, we set 
$$ \bar{d_S}(\Lambda) := \sup_{x \in \Lambda} \bar{d_S}(x)  .$$

\begin{thm}[de Th\'elin-Vigny \cite{dTV}]\label{dTV}
Let $f : \Pj^2 \to \Pj^2$ be a holomorphic endomorphism of degree $d \geq 2$.
Let $S$ be a positive closed current on $\Pj^2$ of bidegree $(1,1)$ and mass $1$. Let $\nu$ be an ergodic measure with positive Lyapunov exponents $\lambda_1 \geq \lambda_2$. Assume that $\supp(\nu) \subset \supp (S)$. Let $\Lambda \subset \supp(\nu)$ be a Borel set such that $\nu(\Lambda) > 0$. Then
$$ \bar{d_S}(\Lambda) \geq 2 {\lambda_2 \over \lambda_1}  + {h_\nu - \log d \over \lambda_1} .$$
\end{thm}

The proof in \cite{dTV} uses delicate slicing arguments to analyse the pullback action of $f^n$ on $\omega$, our Distortion theorem allows to replace these arguments. As before we shall ignore the $e^{\pm n \epsilon}$ error terms due to the non-uniform hyperbolic setting. \\

{\sc Sketch of the proof:} Since $(f^n)^* \omega$ is cohomologous to $d^n \omega$ on $\Pj^2$ we have
\begin{equation}\label{dt1} 
d^n = \int_{\Pj^2} S \wedge (f^n)^* \omega =  \int_{\Pj^2} (f^n)_* S \wedge \omega .
\end{equation}
 Let $\{ x_i \, , \, 1 \leq i \leq N_n \}$ be a $(n,2\eta)$-separated subset of $\Lambda$. Brin-Katok Theorem implies that $N_n \simeq e^{n h_\nu}$. Since the dynamical balls $(B_n(x_i, \eta))_i$ are pairwise disjoint, we get
\begin{equation}\label{dt2}   \int_{\Pj^2} (f^n)_* S \wedge \omega \geq  \sum_{i=1}^{N_n} \int_{\Pj^2} (f^n)_* \left( 1_{B_n(x_i, \eta)} S \right) \wedge \omega =   \sum_{i=1}^{N_n} \int_{B_n(x_i, \eta)} S \wedge (f^n)^* \omega .
\end{equation}
Now we use $B_{x_i} (\eta e^{-n \lambda_1 }) \subset f^{-n}_{{\hat f}^n(\hat x_i)}  (B_{f^n(x_i)}(\eta)) \subset B_n(x_i, \eta)$ and $(f^n)^* \omega \geq e^{2n \lambda_2 } \omega$, which can be proved by using the Distortion Theorem.

Combining this with (\ref{dt1}) and (\ref{dt2}) we obtain
$$ d^n \geq  e^{2n \lambda_2 }  \sum_{i=1}^{N_n}( S\wedge \omega ) ( B_{x_i} (\eta e^{-n \lambda_1}) )  .  $$
From the definition of $\bar{d_S}(x_i)$, and then from the definition of $\bar{d_S}(\Lambda)$ and  $N_n \simeq e^{n h_\nu}$ we get 
$$ d^n  \geq   e^{2n \lambda_2 }   \sum_{i=1}^{N_n}  (\eta e^{-n \lambda_1}) ^{\bar{d_S}(x_i)}    \geq  e^{2n \lambda_2 }   e^{n h_\nu}  (\eta e^{-n \lambda_1}) ^{\bar{d_S}(\Lambda)} . $$
The comparison of the exponential growth rates gives the desired lower bound on ${\bar{d_S}(\Lambda)}$.

\section{Appendix}\label{app}

\subsection{Bundle maps}\label{bundle}

Let us recall that  $\CC$ is the critical set of $f$, that 
$$X = \{ \hat x = (x_n)_{n
  \in \Z} \, \colon \,  x_{n+1}=f(x_n) \ , \ x_n \notin \CC \ , \
\forall n \in \Z \}  $$ 
and that  $\tau : X \to X$ is the right shift sending 
  $(\cdots,x_{-1},x_0,x_1,\cdots)$ to
$(\cdots,x_{-2},x_{-1},x_0,\cdots)$. We denote $$E  := \cup_{\hat x \in X}  \{ \hat x \} \times \C^k  $$
and $E_{\hat x} :=  \{ \hat x \} \times \C^k$. For every positive real number $a >0$ we denote  $E_{\hat x}(a) := \{ \hat x \} \times \D^k(a)$.  
More generally,  for every positive function $a : X \to \R^+_*$, we let $E_{\hat x}(a) := E_{\hat x}(a(\hat x))$ and 
$$E(a) := \cup_{\hat x \in X} E_{\hat x}(a)     =  \cup_{\hat x \in X}  \{ \hat x \} \times \D^k(a (\hat x)) .$$
Let $\si \in \{ \Id_X , \tau \}$ and $a, b : X \to \R^+_*$ be two positive functions. A \emph{bundle map $\KK : E(a) \to E(b)$ over $\si$} is a map of the form 
$$\KK(\hat x,v) = (\si(\hat x) , K_{\hat x}(v)), $$ where 
$$  \ K_{\hat x} : E_{\hat x}(a(\hat x)) \to E_{\si(\hat x)}(b (\si(\hat x)))$$ 
is holomorphic and satisfies  $K_{\hat x}(0)=0$ for every $ \hat x \in X$. 
The \emph{linear tangent} bundle map $d_0\KK : E \to E$ is defined by 
$$d_0\KK (\hat x,v) = (\sigma(\hat x) , d_0 K_{\hat x} (v)). $$ We say that $\KK$ is \emph{tame} if  there exist a $\epsi$-slow function
$r_\epsi$ and a $\epsi$-fast function $s_\epsi$ satisfying $$\KK : E(r_\epsi) \to E(s_\epsi) .$$

\subsection{Results on tame bundle maps}\label{tamebundle}

The following lemma  simply relies on Cauchy's estimates, see \cite[Lemma 2.3]{BDM}. It implies in particular that if $\KK$ is tame with a contracting linear part, then   $\KK : E(a_\epsi) \to E(a_\epsi)$ for some $\epsilon$-slow function $a_\epsi$ and thus $\KK$ can be iterated.  

\begin{lem}\label{NN}
Let $\si \in \{ \Id_X, \tau \}$, let $\epsilon >0$ and let $\KK$ be a tame bundle map over $\si$. Assume that there exist $0< \alpha \leq \beta$ such that  $$ \forall {\hat x} \in X \ \ , \ \ \forall v \in \C^k \ \ , \ \ \alpha \abs{v} \leq \abs{d_0 K_{\hat x}(v)} \leq  \beta \abs{v} .$$ 
Then the following estimates occur.
\begin{enumerate}
\item For every $\kappa > 0$, there exists a $\epsi$-slow function $\phi_\epsi : X \to ]0,1]$ such that 
$$\forall \hat x \in X \ , \   \Lip (K_{\hat x} - d_0 K_{\hat x}) \leq \kappa \textrm{ on  }  E_{\hat x}(\phi_\epsi) .$$

\item For every $\epsi' >0$, there exists a $\epsi$-slow function $\phi_\epsi : X \to ]0,1]$ such that 
$$\forall \hat x \in X \ , \  \forall (u,v) \in E_{\hat x}(\phi_\epsi) \times  E_{\hat x}(\phi_\epsi) \ , \ \alpha e^{-\epsi '} \abs{u-v} \leq
  \abs{K_{\hat x}(u) - K_{\hat x}(v)} \leq \beta e^{\epsi '} \abs{u-v} . $$

In particular, if $\beta < 1$ (contracting case) and if $\beta e^{\epsi '} \leq e^{-\epsi}$, then 
$$\KK : E(\psi_\epsi) \to E(\psi_\epsi) $$ 
for every $\epsilon$-slow function $\psi_\epsi$ satisfying $\psi_\epsi \leq \phi_\epsi$.

\end{enumerate}
\end{lem}

\begin{proof}
Let $\epsilon, \epsilon' , \kappa >0$. Since $\KK$ is tame, there exist a $\epsi/3$-slow (resp. fast) function $r_{\epsi/3}$ (resp. $s_{\epsi/3})$  such that $\KK : E(r_{\epsi/3}) \to E(s_{\epsi/3})$. Let $\hat x \in X$. Cauchy's
estimates on $E_{\hat x}({1 \over 2} r_{\epsi/3})$ bound the second derivatives of $K_{\hat x}$ by $c s_{\epsi/3}(\si(\hat x)) /r_{\epsi/3}(\hat x)^2$ where the  constant $c$ only depends  on the dimension $k$. We deduce that for every $\rho \leq {1 \over 2} r_{\epsi/3}(\hat x)$:
\begin{equation}\label{aze}
 \forall t \in E_{\hat x}(\rho) \ , \  \abs{d_t (d_0 K_{\hat x}  - K_{\hat x})} = \abs{d_0 K_{\hat x} - d_t K_{\hat x}} \leq   {  c s_{\epsi/3}(\si(\hat x)) \over r_{\epsi/3}(\hat x)^2} \rho.
\end{equation}
Let us define
\[ \phi_\epsi :=  {r_{\epsi/3}^2 \over c  s_{\epsi/3}\circ \si}  \min \{ (e^{\epsi'} - 1) \beta \, , \, (1-e^{-\epsi'}) \alpha \, , \, \kappa   \} , \]
which is a $\epsi$-slow function. Item 1 is then a consequence of (\ref{aze}). To verify Item 2 we put $\rho = \phi_\epsi$ in (\ref{aze}) to obtain the following estimates on $E_{\hat x}(\phi_\epsi)$: 
\begin{displaymath}
\begin{array}{rcl}
   \abs{K_{\hat x}(u) - K_{\hat x}(v)}   & \leq  & \abs{ d_0 K_{\hat x} (u) - d_0 K_{\hat x} (v) } + \abs {(d_0 K_{\hat x} 
  - K_{\hat x})(u) - (d_0 K_{\hat x}  - K_{\hat x})(v)}    \\
                   &   \leq   &  \beta \abs{u-v} +  {c s_{\epsi/3}(\si(\hat x)) \over r_{\epsi/3}(\hat x)^2}
  \phi_\epsi(\hat x) \abs{u-v} \\
                   & \leq   &  \beta e^{\epsi'}  \abs{u-v}. 
\end{array}                                      
\end{displaymath} 
We obtain similarly $\abs {K_{\hat x}(u) - K_{\hat x}(v)} \geq \alpha e^{-\epsi'}
\abs{u-v}$. If $\beta e^{\epsi'} \leq e^{-\epsi}$ and if  $\psi_\epsi$ is a $\epsilon$-slow function satisfying $\psi_\epsi \leq \phi_\epsi$  then $\abs {K_{\hat x}(u)} \leq \beta e^{\epsi'} \psi_\epsi(\hat x) \leq e^{-\epsi} \psi_\epsi(\hat x) \leq  \psi_\epsi(\si(\hat x))$ on $E(\psi_\epsi)$. \fin
 \end{proof} 
 
The next lemma is useful for conjugating bundle maps, it is a corollary of Lemma \ref{NN}. 

\begin{lem} \label{NNbis}
Let $\MM$ be a tame bundle map over $\Id_X$ and let $\LL$ be a tame bundle map over $\tau$.
We assume that $d_0 \MM = (\Id_X , \Id_{\C^k})$ and that there exist $0< \alpha \leq \beta < 1$ such that 
$$ \forall {\hat x} \in X \ \ , \ \ \forall v \in \C^k \ \ , \ \ \alpha \abs{v} \leq \abs{d_0 L_{\hat x}(v)} \leq  \beta \abs{v} .$$ 
Let $0 < \epsi' < \epsi$ such that $\beta e^{\epsi'} < e^{-\epsi}$. Then there exists a $\epsilon$-slow function $\phi_\epsi$ such that for every $\epsilon$-slow function $\psi_\epsi \leq \phi_\epsi$,

\begin{enumerate}
\item the bundle map $\tilde \LL := \MM \circ \LL \circ \MM^{-1}$ is well defined on $E(e^\epsi \psi_\epsi)$.
\item the following diagram commutes:
\[  \xymatrix{
  E(\psi_\epsi) \ar[rrr]^{ \LL  } \ar[d]_{\MM}& & & E(\psi_\epsi)   \ar[d]_{\MM}  \\
    E(e^\epsi \psi_\epsi)  \ar[rrr]^{ \tilde \LL }  & & & E(e^\epsi \psi_\epsi)  . }  \]
 \item $\forall \hat x \in X$, $\forall (u,v) \in E_{\hat x}(\psi_\epsi) \times  E_{\hat x}(\psi_\epsi)$, $e^{-\epsi '} \abs{u-v} \leq
  \abs{M_{\hat x}(u) - M_{\hat x}(v)} \leq e^{\epsi '} \abs{u-v}$ .

\end{enumerate}
\end{lem}

\begin{proof}
Let $0 < \epsi' < \epsi$ such that $\beta e^{\epsi'} < e^{-\epsi}$. Let $\phi^1_\epsi$ be a $\epsi/3$-slow function provided by Item 2 of Lemma \ref{NN} such that 
$$\forall (u,v) \in E_{\hat x}(\phi^1_\epsi) \times E_{\hat x}(\phi^1_\epsi) \ , \ 
  \abs{L_{\hat x}(u) - L_{\hat x}(v)} \leq \beta e^{\epsi '} \abs{u-v} \leq e^{- \epsi} \abs{u-v} . $$
In particular we have $\LL  : E(\phi^1_\epsi) \to E(\phi^1_\epsi)$.  
Let $\phi^2_\epsi$ be a $\epsi/3$-slow function provided by the same lemma such that  
$$\forall (u,v) \in E_{\hat x}(\phi^2_\epsi) \times \in E_{\hat x}(\phi^2_\epsi) \ , \ e^{-\epsi '} \abs{u-v} \leq
  \abs{M_{\hat x}(u) - M_{\hat x}(v)} \leq e^{\epsi '} \abs{u-v} , $$
hence we have  $\MM  : E(\phi^2_\epsi) \to E(e^\epsi \phi^2_\epsi)$.
We set $\phi_\epsi := \min \{ \phi^1_\epsi, \phi^2_\epsi \}$, which is a $\epsi$-slow function.
Then $\tilde \LL := \MM \circ \LL \circ \MM^{-1}$ is well defined on $E(e^{-\epsi}\phi_\epsi)$ and takes its values in $E(e^{\epsi}\phi_\epsi)$. Since $d_0 \tilde \LL = d_0 \LL$ is contracting, we can replace $\phi_\epsi$ by a smaller $\epsi$-slow function to have $\tilde \LL : E(e^{\epsi}\phi_\epsi) \to E(e^{\epsi}\phi_\epsi)$. All these properties obviously hold for every $\epsilon$-slow function $\psi_\epsi \leq \phi_\epsi$.  \finsec
\end{proof}

\subsection{Constraints on $\gamma$, $\epsilon$ and cancellation of resonances}\label{gamma}

We use here the notations introduced in Subsection \ref{RC}. Let us recall that  $0 < a < \ln 4$ is fixed such that
\begin{equation}\label{rol2} 
 \al_1\lambda_1 +\cdots +\al_k \lambda_k  - \Lambda_j  \notin [-a,a] 
\end{equation}
holds for every $j \in \{ 1 , \cdots , l \}$ and for every $\al \in \N^k\setminus \RRR_j$  satisfying $2 \leq \abs{\al} \leq [2 \Lambda_1 / \Lambda_l]$. Let us  also recall the Constraints \ref{nota}.

\begin{itemize} 
\item[1.]  The number  $\gamma >0$ is fixed and sufficently small so that:

$\gamma < \Lambda_l / 2 \ \ , \ \ \gamma ( [ \Lambda^\gamma_1 / \Lambda^\gamma_l ] -1)< a/2 \ \ \textrm{ and }  \ Ê\   4 \gamma (\Lambda^\gamma_1 / \Lambda^\gamma_l +1) \leq \Lambda_l^\gamma .$ 

\item[2.]   Any choice of $\epsilon >0$ is supposed to be small enough so that:

$2\epsilon < \gamma \ \ , \ \  4\epsilon + 2\gamma  < \Lambda_l \ \ \textrm{ and }  \ Ê\    \epsi  ( [\Lambda^\gamma_1 / \Lambda^\gamma_l] + 3 )  < {1 \over 2 } \min \{ \gamma , a  \} =: b . $
\end{itemize}
We now prove two elementary results.

\begin{lem}\label{donne} $ $
\begin{enumerate}
\item For every  $j \in \{ 1 , \cdots , l \}$ and every $\al \in \N^k \setminus\RRR_j$ such that $2 \leq \abs{\al} \leq [\Lambda^\gamma_1 / \Lambda^\gamma_l]$, we have
$$ \al_1\lambda_1 +\cdots +\al_k \lambda_k  - \Lambda_j  \notin [-a,a] . $$  
\item If $q := [\Lambda^\gamma_1 / \Lambda^\gamma_l] +1$,  $m := e^{-\Lambda^\gamma_1 - \epsilon + \gamma}$ et $M  := e^{-\Lambda^\gamma_l + \epsilon + \gamma}$, then  $(M e^{2\epsi})^{q+1} < m e^{-\epsilon}$.
\end{enumerate}
\end{lem}

\begin{proof}
The first statement immediately follows from (\ref{rol2}) after observing that $\gamma < \Lambda_l / 2$  yields $[ \Lambda^\gamma_1 / \Lambda^\gamma_l ] \leq [2 \Lambda_1 / \Lambda_l]$. For the second statement, one has
$$ ( - \Lambda^\gamma_l + 3\epsilon + \gamma ) ( [\Lambda^\gamma_1 / \Lambda^\gamma_l] +2) \leq  ( - \Lambda^\gamma_l + 3 \epsilon + \gamma)  (\Lambda^\gamma_1 / \Lambda^\gamma_l +1) = - \Lambda_1^\gamma - \big\{ \Lambda_l^\gamma  -  (3\epsi + \gamma) (\Lambda^\gamma_1 / \Lambda^\gamma_l +1)   \big\}  $$
and  $\big\{\;\;\big\}\geq \Lambda_l^\gamma  -  4\gamma (\Lambda^\gamma_1 / \Lambda^\gamma_l +1)   \geq 0 >  2\epsilon - \gamma$. \fin 
\end{proof}

\begin{lem}\label{marge}
For every $j \in \{ 1 , \cdots , l \}$ and every  $2 \leq \abs{\al} \leq [\Lambda^\gamma_1 / \Lambda^\gamma_l]$, one has
$$ \al_1 \lambda^\gamma_1 +\cdots +\al_k \lambda^\gamma_k - \Lambda^\gamma_j  \notin [-b , b] . $$ 
In particular, the $\lambda_i^\gamma$'s do not satisfy any resonant relation: $\cup_{j=1}^l  \RRR_j^\gamma = \emptyset$.
 \end{lem}

\begin{proof}
Let us fix $j \in \{ 1 , \cdots , l \}$ et $\alpha \in \N^k$ such that $2 \leq \abs{\al} \leq [\Lambda^\gamma_1 / \Lambda^\gamma_l]$. We have
$$ \al_1 \lambda^\gamma_1 +\cdots +\al_k \lambda^\gamma_k - \Lambda^\gamma_j = \alpha \cdot \lambda - \gamma \vert \alpha \vert - (\Lambda_j - \gamma) = \alpha \cdot \lambda - \Lambda_j - \gamma (\vert \alpha \vert -1) . $$
Assume first that $\al \in \RRR_j$. Since $\alpha \cdot \lambda - \Lambda_j = 0$ and $\vert \alpha \vert \geq 2$, one has 
$$ \al_1 \lambda^\gamma_1 +\cdots +\al_k \lambda^\gamma_k - \Lambda^\gamma_j  =  - \gamma (\vert \alpha \vert -1) \leq - \gamma < - b.$$
Let us now assume that  $\al \notin \RRR_j$. We use here the first assertion of Lemma \ref{donne}.\\
 If $\alpha \cdot \lambda - \Lambda_j > a$ then,  as $\abs{\al} \leq [\Lambda^\gamma_1 / \Lambda^\gamma_l]$, one gets
$$ \al_1 \lambda^\gamma_1 +\cdots +\al_k \lambda^\gamma_k - \Lambda^\gamma_j > a - \gamma  ([\Lambda^\gamma_1 / \Lambda^\gamma_l] - 1) > a/2 \geq b . $$
 If  $\alpha \cdot \lambda - \Lambda_j < -a$ then, as  $\abs{\al} \geq 2$, one gets
$$ \al_1 \lambda^\gamma_1 +\cdots +\al_k \lambda^\gamma_k - \Lambda^\gamma_j <  - a -\gamma (\vert \alpha \vert -1) \leq -a < - b . $$
This completes the proof of the lemma. \finsec
\end{proof}

\vspace{1 cm}

{\footnotesize F. Berteloot}\\
{\footnotesize Universit\'e Toulouse 3}\\
{\footnotesize Institut Mathématique de Toulouse}\\
{\footnotesize Equipe Emile Picard, B\^at. 1R2}\\
{\footnotesize 118, route de Narbonne }\\
{\footnotesize F-31062 Toulouse Cedex 9, France}\\
{\footnotesize berteloo@picard.ups-tlse.fr}\\ 

{\footnotesize C. Dupont}\\
{\footnotesize Universit\'e de Rennes 1}\\
{\footnotesize IRMAR, CNRS UMR 6625}\\
{\footnotesize Campus de Beaulieu, B\^at. 22-23}\\
{\footnotesize F-35042 Rennes Cedex, France}\\
{\footnotesize christophe.dupont@univ-rennes1.fr}\\


\begin{thebibliography}{<100>}


\bibitem[B]{Be} F. Berteloot, \it{Bifurcation currents in holomorphic families of rational maps}, \rm 
 Lecture Notes in Mathematics {\bf 2075} CIME Fundation subseries \rm(2013)
Springer Verlag, 1-93.


\bibitem[BB]{BB} F. Berteloot, F. Bianchi, {\it Perturbations d'exemples de Latt\`es et dimension de Hausdorff du lieu de bifurcation}, to appear in {J. Math. Pures et Appl.}

\bibitem[BBD]{BBD} F. Berteloot, F. Bianchi, C. Dupont, {\it Dynamical stability and Lyapunov exponents for holomorphic endomorphisms of $\Pj^k$}, \rm to appear in Ann. Sci. Ecole Norm. Sup. 

\bibitem[BDM]{BDM} F. Berteloot, C. Dupont, L. Molino, {\it Normalization of bundle holomorphic contractions and applications to dynamics}, \rm Ann. Inst. Fourier, {\bf 58} \rm(2008), no. 6, 2137-2168.

\bibitem[BDeM]{BDeM} I. Binder, L. DeMarco, \it{Dimension of pluriharmonic measure and polynomial  endomorphisms of ${\bf C}^n$}, \rm Int. Math. Res. Not., {\bf 11} \rm(2003), 613-625.


\bibitem[BD]{BD} J.Y. Briend, J. Duval, \it{Exposants de Liapounoff et distribution des points p\'eriodiques d'un endomorphisme de $\C \Pj^k$}, \rm Acta Math., \bf{182} \rm (1999), no. 2, 143-157. 

\bibitem[CFS]{CFS} I.P. Cornfeld, S.V. Fomin, Ya. B. Sinai, \it {Ergodic theory}, \rm Grund. Math. Wiss. No 245, Springer, 1985.

\bibitem[D]{D} J.-P. Demailly, \it {Complex analytic and differential geometry}, \rm available at https://www-fourier.ujf-grenoble.fr/~demailly/documents.html.


\bibitem[dT]{dT} H. de Th\'elin, \it{Sur les exposants de Lyapounov des applications m\'eromorphes}, \rm Invent. Math. \bf{172} \rm (2008),  no. 1, 89-116.

\bibitem[dTV]{dTV} H. de Th\'elin, G. Vigny, \it{On the measures of large entropy on a positive closed current}, \rm Math. Z. \bf{280} \rm (2015), 919-944.

\bibitem[Di]{Di} T.-C. Dinh, \it{Attracting current and equilibrium measure for attractors on $\Bbb P\sp k$}, \rm  J. Geom. Anal.  \bf{17} \rm (2007),  no. 2, 227-244.

\bibitem[DS]{DS} T.-C. Dinh, N. Sibony, \it{Dynamics in several complex variables: endomorphisms of projective spaces and polynomial-like mappings}, \rm Lecture Notes in Math. \bf{1998} \rm (2010).

\bibitem [D1]{D1}  C. Dupont, \it{On the dimension of invariant measures of endomorphisms of $\Pj^k$}, \rm  Math. Ann. \bf{349} \rm(2011), 509-528.

\bibitem [D2]{D2}  C. Dupont, \it{Large entropy measures for endomorphisms of $\C\Pj^k$}, \rm Israel J. Math. \bf{192} \rm (2012), 505-533.

\bibitem[GK]{GK} M. Guysinsky, A. Katok, \it{Normal forms and invariant geometric structures for dynamical systems with invariant contracting foliations}, \rm Math. Res. Lett., {\bf 5} \rm (1998), no. 1-2, 149-163. 


\bibitem[JV]{JV} M. Jonsson, D. Varolin, \it{Stable manifolds of holomorphic diffeomorphisms}, \rm Invent. Math. {\bf 149} \rm (2002), no. 2, 409-430.

\bibitem[KH]{KH} A. Katok, B. Hasselblatt, \it{Introduction to the modern theory of dynamical systems}, \rm Cambridge Univ. Press, 1995.

\bibitem[KS]{KS} A. Katok, R. Spatzier, \it{Nonstationary normal forms and rigidity of group actions}, \rm Electron. Res. Announc. Amer. Math. Soc., \bf{2} \rm (1996), no. 3, 124-133. 

\bibitem[L]{L} M. Lyubich, \it{Entropy properties of rational endomorphisms of the Riemann sphere}, \rm Ergodic Theory Dynamical Systems, \bf{3} \rm (1983), no. 3, 351-385.


\bibitem[M]{M} R. Ma\~n\'e, \it{The Hausdorff dimension of invariant probabilities of rational maps}, \rm Lecture Notes in Math., {\bf 1331}, Springer, 1988. 

\bibitem[Y]{Y} L.S. Young, \it{Dimension, entropy and Lyapounov exponents}, \rm Ergodic Theory \& Dynamical Systems, \bf{2} \rm (1982), no. 1, 109-124.

\end{thebibliography}
\end{document}